\documentclass[smallcondensed]{svjour3}
\smartqed
\usepackage{graphicx,subfig}
\usepackage{stmaryrd}
\usepackage{commath}
\usepackage{hyperref}

\usepackage{cleveref}
\usepackage{url}
\usepackage{amsmath}
\usepackage{amssymb}
\usepackage[numbers]{natbib}

\DeclareFontFamily{U}{matha}{\hyphenchar\font45}
\DeclareFontShape{U}{matha}{m}{n}{
      <5> <6> <7> <8> <9> <10> gen * matha
      <10.95> matha10 <12> <14.4> <17.28> <20.74> <24.88> matha12
      }{}
\DeclareSymbolFont{matha}{U}{matha}{m}{n}
\DeclareFontSubstitution{U}{matha}{m}{n}

\DeclareFontFamily{U}{mathx}{\hyphenchar\font45}
\DeclareFontShape{U}{mathx}{m}{n}{
      <5> <6> <7> <8> <9> <10>
      <10.95> <12> <14.4> <17.28> <20.74> <24.88>
      mathx10
      }{}
\DeclareSymbolFont{mathx}{U}{mathx}{m}{n}
\DeclareFontSubstitution{U}{mathx}{m}{n}

\DeclareMathDelimiter{\vvvert}{0}{matha}{"7E}{mathx}{"17}

\newcommand{\tnorm}[1]{{\left\vvvert #1 \right\vvvert}}

\journalname{Journal of Scientific Computing}
\begin{document}
%------------------------------------------------------------------------------

\title{Preconditioning of a hybridized discontinuous Galerkin finite
  element method for the Stokes equations\thanks{SR gratefully
    acknowledges support from the Natural Sciences and Engineering
    Research Council of Canada through the Discovery Grant program
    (RGPIN-05606-2015) and the Discovery Accelerator Supplement
    (RGPAS-478018-2015).} }

\titlerunning{Preconditioning of an HDG method for the Stokes problem}

\author{Sander Rhebergen \and Garth N.~Wells}

\institute{Sander Rhebergen \at
  Department of Applied Mathematics, University of Waterloo,
  Waterloo N2L~3G1, Canada \\
  \email{srheberg@uwaterloo.ca} \\
  ORCID:~0000-0001-6036-0356
  \and
  Garth N.~Wells \at
  Department of Engineering, University of Cambridge,
  Trumpington Street, Cambridge CB2~1PZ, United Kingdom \\
  \email{gnw20@cam.ac.uk} \\
  ORCID:~0000-0001-5291-7951
}

\date{}

\maketitle

\begin{abstract}
  We present optimal preconditioners for a recently introduced
  hybridized discontinuous Galerkin finite element discretization of
  the Stokes equations. Typical of hybridized discontinuous Galerkin
  methods, the method has degrees-of-freedom that can be eliminated
  locally (cell-wise), thereby significantly reducing the size of the
  global problem. Although the linear system becomes more complex to
  analyze after static condensation of these element
  degrees-of-freedom, the pressure Schur complement of the original
  and reduced problem are the same. Using this fact, we prove spectral
  equivalence of this Schur complement to two simple matrices, which
  is then used to formulate optimal preconditioners for the statically
  condensed problem.  Numerical simulations in two and three spatial
  dimensions demonstrate the good performance of the proposed
  preconditioners.

\end{abstract}

%------------------------------------------------------------------------------
\section{Introduction}
\label{sec:introduction}

Recently, many hybridized discontinuous Galerkin (HDG) methods have
been introduced for incompressible flows. For the Stokes problem these
include \cite{Cockburn:2009c, Cockburn:2014b, Cockburn:2011,
  Nguyen:2010}, and for the Oseen and Navier--Stokes problems we refer
to \cite{Cesmelioglu:2013, Cesmelioglu:2017, Labeur:2012,
  Lehrenfeld:2016, Nguyen:2011, Qiu:2016, Rhebergen:2012}. We consider
the method developed in \cite{Labeur:2012} for the Navier--Stokes
equations, but with tighter restrictions on the `facet' function
spaces.  The method is appealing in its simplicity and the fact that
it can be formulated such that the approximate velocity field is
automatically pointwise divergence-free. However, the implementation
in \cite{Labeur:2012} does not yield a $H({\rm div})$-conforming
velocity field and, as consequence, cannot be simultaneously locally
mass conserving, locally momentum conserving and energy stable. This
issue was resolved in~\cite{Rhebergen:2017,Rhebergen:2018}, in which
the method was modified for the Stokes and Navier--Stokes equations
such that the approximate velocity fields are both pointwise
divergence-free and $H({\rm div})$-conforming.  This leads to a method
that is locally mass conserving, momentum conserving, energy stable,
and \emph{pressure-robust}~\cite{John:2017}, as shown numerically
in~\cite{Rhebergen:2018}. For the method in
\cite{Rhebergen:2017,Rhebergen:2018} to be useful in practice, it is
helpful if the discrete system arising from the method can be solved
efficiently by iterative methods. In this work we introduce and
analyze new preconditioners for the method applied to the Stokes
problem, and show that optimal preconditioners can be constructed.

A feature of the HDG approach is static condensation; element degrees of
freedom can be eliminated locally from the linear system, thereby
significantly reducing the size of the global problem. There are
different ways to apply static condensation to the HDG method
of~\cite{Rhebergen:2017}. One may choose, for example, to eliminate
both the element velocity and element pressure degrees-of-freedom. In
this paper, however, we choose to eliminate only the element velocity
degrees-of-freedom. In terms of reducing the global problem size, the
effect of eliminating the element pressure degrees-of-freedom is
minimal, whereas the reduction in global system size when eliminating
the element velocity degrees-of-freedom is substantial. We do not
consider elimination of the pressure degrees-of-freedom on cells as
retaining the cell pressure field will lead to a formulation on which
standard multigrid methods may be applied in the construction of optimal
preconditioners. This would not be possible if cell pressure
degrees-of-freedom are also eliminated from the original system.

The linear system obtained after static condensation of the cell-wise
velocity degrees-of-freedom is more complex to analyze than the original
linear system. However, the element/facet pressure Schur complement
remains unchanged. Using boundedness and stability results
of~\cite{Rhebergen:2017}, and a suitable inf-sup condition, we prove
spectral equivalence between the element/facet pressure Schur complement
and an element/facet pressure mass matrix. This allows the general
theory of~\citet{Pestana:2015} for preconditioners for saddle point
problems to be applied, which we use to develop two new preconditioners
for the condensed HDG discretization of the Stokes equations. Optimality
of the preconditioners for the HDG problem is proved, and numerical
examples demonstrate very good performance.

The remainder of this paper is structured as follows. In
\cref{sec:stokes} we describe the HDG method for the Stokes equations
and discuss and prove boundedness and stability results. These results
are then used to develop and analyze preconditioners for the condensed
form of the HDG discretization in \cref{s:preconditioning}. We verify
our analysis by two- and three-dimensional numerical simulations in
\cref{s:numerical_examples} and provide conclusions in
\cref{sec:conclusions}.

%------------------------------------------------------------------------------
\section{Hybridizable discontinuous Galerkin method: formulation and
  analysis}
\label{sec:stokes}

We consider the Stokes system:
\begin{subequations}
  \label{eq:stokes}
  \begin{align}
    -\nabla^2u + \nabla p &= f & & \mbox{in} \ \Omega,
    \\
    \nabla\cdot u &= 0 & & \mbox{in}\ \Omega,
    \\
    u &= 0 & & \mbox{on}\ \partial\Omega,
    \\
    \int_{\Omega} p \dif x &= 0,
  \end{align}
\end{subequations}
where $\Omega\subset\mathbb{R}^d$ is a polygonal ($d = 2$) or polyhedral
($d = 3$) domain, $u : \Omega \to \mathbb{R}^d$ is the velocity, $p :
\Omega \to \mathbb{R}$ is the pressure, and $f : \Omega \to
\mathbb{R}^d$ is a prescribed body force.

%------------------------------------------------------------------------------
\subsection{Notation}

To define the hybridizable discontinuous Galerkin method for the
Stokes equations, we introduce first a triangulation $\mathcal{T} :=
\{K\}$ of $\Omega$ consisting of non-overlapping cells. Each cell $K$
of the triangulation has a length measure $h_K$, and on the boundary
of an element, $\partial K$, the outward unit normal vector is denoted
by~$n$.  Two adjacent cells $K^+$ and $K^-$ share an interior facet
$F$, while a boundary facet is a facet of $\partial K$ that lies on
$\partial\Omega$.  The set and union of all facets are denoted by
$\mathcal{F} = \{F\}$ and $\Gamma^0$, respectively.

We will use the following finite element function spaces on~$\Omega$:
\begin{equation}
  \begin{split}
    V_h  &:= \cbr{v_h\in \sbr{L^2(\Omega)}^d
        : \ v_h \in \sbr{P_k(K)}^d, \ \forall\ K\in\mathcal{T}},
    \\
    Q_h &:= \cbr{q_h\in L^2(\Omega) : \ q_h \in P_{k-1}(K) ,\
    \forall \ K \in \mathcal{T}},
  \end{split}
  \label{eqn:spaces_cell}
\end{equation}
and the following finite element spaces on $\Gamma^0$,
\begin{equation}
  \begin{split}
    \bar{V}_h &:= \cbr{\bar{v}_h \in \sbr{L^2(\Gamma^0)}^d:\ \bar{v}_h \in
      \sbr{P_{k}(F)}^d\ \forall\ F \in \mathcal{F},\ \bar{v}_h
      = 0 \ \mbox{on}\ \partial\Omega},
    \\
    \bar{Q}_h &:= \cbr{\bar{q}_h \in L^2(\Gamma^0) : \ \bar{q}_h \in
      P_{k}(F) \ \forall\ F \in \mathcal{F}},
  \end{split}
  \label{eqn:spaces_facet}
\end{equation}
where $P_k(D)$ denotes the set of polynomials of degree at most~$k$ on a
domain~$D$. For convenience, we introduce the spaces $V_h^{\star} := V_h
\times \bar{V}_h$, $Q_h^{\star} := Q_h \times
\bar{Q}_h$, and $X_h^{\star} := V_h^{\star} \times Q_h^{\star}$.
Function pairs in $V_h^{\star}$ and $Q_h^{\star}$ will be denoted by
boldface, e.g., ${\bf v}_h := (v_h, \bar{v}_h) \in V_h^{\star}$ and
${\bf q}_h := (q_h, \bar{q}_h) \in Q_h^{\star}$.

On an element $K \subset \mathbb{R}^d$, for scalar functions $p, q \in
L^2(K)$, we denote the standard  inner-product by $(p, q)_K := \int_{K}
p q \dif x$, and  we define $(p, q)_{\mathcal{T}} :=
\sum_{K \in \mathcal{T}}(p, q)_K$. For scalar functions $p, q \in
L^2(E)$, where $E \subset \mathbb{R}^{d - 1}$, we define the
inner-product $\langle p, q \rangle_E := \int_E p q \dif s$ and $\langle
p, q
\rangle_{\partial\mathcal{T}} := \sum_K \langle p, q
\rangle_{\partial K}$. Similar inner-products hold for vector-valued
functions.

We use various norms throughout, and which are defined now. On $V_h$ and
$V_h^{\star}$ we define, respectively, the following `discrete'
$H^1$-norms:
\begin{align}
  \label{eq:stability_norm_dg}
  \tnorm{ v_h }_{DG}^2 &:= \sum_{K\in\mathcal{T}}\norm{\nabla v_h }^2_{K}
  + \sum_{K\in\mathcal{T}} \alpha h_K^{-1} \norm{v_h}^2_{\partial K},
  \\
  \label{eq:stability_norm}
  \tnorm{ {\bf v}_h }_v^2 &:= \sum_{K\in\mathcal{T}}\norm{\nabla v_h }^2_{K}
  + \sum_{K\in\mathcal{T}} \alpha  h_K^{-1} \norm{\bar{v}_h - v_h}^2_{\partial K},
\end{align}
where $\alpha > 0$ is a constant. For $\bar{v}_h \in \bar{V}_h$, we
introduce the norm
\begin{equation}
  \label{eq:def_norm_hK_ap}
  \tnorm{\bar{v}_h}_{h}^2 := \sum_{K\in\mathcal{T}_h}h_{K}^{-1}
    \norm{\bar{v}_h - m_K(\bar{v}_h)}_{\partial K}^2,
\end{equation}
where
\begin{equation}
  \label{eq:mean_pK}
  m_K(\bar{v}_h) := \frac{1}{|\partial K|}\int_{\partial K} \bar{v}_h \dif s.
\end{equation}
On $\bar{Q}_h$ and $Q^{\star}_h$ we define, respectively, `discrete'
$L^2$-norms,
\begin{equation}
  \label{eq:def_normbarq}
    \norm{\bar{q}_h}_{p}^2 := \sum_{K\in\mathcal{T}}h_K\norm{\bar{q}_h}^2_{\partial K}
    \quad \text{and} \quad
    \tnorm{{\bf q}_h}^2_p := \norm{q_h}^2_{\Omega} + \norm{\bar{q}_h}_{p}^2.
\end{equation}

%------------------------------------------------------------------------------
\subsection{Weak formulation}

The weak formulation for the Stokes problem in \cref{eq:stokes} is given
in \cite{Labeur:2012,Rhebergen:2017}, and reads: given $f \in
\sbr{L^{2}(\Omega)}^{d}$, find $({\bf u}_h, {\bf p}_h)
\in X_{h}^{\star}$ such that
\begin{subequations}
  \begin{align}
    \label{eq:discrete_problem_a}
    a_h({\bf u}_h, {\bf v}_h) + b_h({\bf p}_h, {\bf v}_h)
    &=
    \del{v_h, f}_{\mathcal{T}} && \forall {\bf v}_h\in V_h^{\star},
    \\
    \label{eq:discrete_problem_b}
    b_h({\bf q}_h, {\bf u}_h) &= 0 && \forall {\bf q}_h\in Q_h^{\star},
  \end{align}
  \label{eq:discrete_problem}
\end{subequations}
where
\begin{subequations}
  \begin{align}
    \label{eq:formA}
    a_h({\bf w}_h, {\bf v}_h)
    :=&
    \del{\nabla w_h, \nabla v_h}_{\mathcal{T}}
    + \left\langle \alpha h^{-1}(w_h - \bar{w}_h), v_h - \bar{v}_h \right\rangle_{\partial\mathcal{T}}
    \\
    \nonumber
    & - \left\langle w_h - \bar{w}_h, \partial_n v_h \right\rangle_{\partial\mathcal{T}}
      - \left\langle \partial_n w_h, v_h - \bar{v}_h \right\rangle_{\partial\mathcal{T}},
    \\
    \label{eq:formB}
    b_h({\bf q}_h, {\bf v}_h)
    :=&
    - \del{q_h, \nabla \cdot v_h}_{\mathcal{T}}
    + \left\langle v_h \cdot n, \bar{q}_h \right\rangle_{\partial \mathcal{T}}.
  \end{align}
  \label{eq:bilin_forms}
\end{subequations}
It is proven in \cite{Rhebergen:2017,Wells:2011} that $\alpha$ can be
chosen sufficiently large to ensure stability.

The formulation is a hybridized method in the sense that the facet
function $\Bar{p}_{h}$ acts as a Lagrange multiplier enforcing that the
velocity field $u_{h}$ is $H({\rm div})$-conforming, and specifically
lies in a Brezzi--Douglas--Marini (BDM) finite element space
\cite{Boffi:book}.

The following results are from \cite{Rhebergen:2017} and will be used
in the analysis. For sufficiently large $\alpha$, the bilinear form
$a_h(\cdot, \cdot)$ is coercive and bounded, i.e., there exist constants
$c_a^s > 0$ and $c_a^b > 0$, independent of $h$, such that for all ${\bf
u}_h, {\bf v}_h \in V_h^{\star}$,
\begin{equation}
  \label{eq:stab_bound_ah}
  a_h({\bf v}_h, {\bf v}_h) \ge c_a^s\tnorm{{\bf v}_h}_v^2 \qquad\mbox{and}\qquad
  \envert{a_h({\bf u}_h, {\bf v}_h)} \le c_a^b\tnorm{{\bf u}_h}_v\tnorm{{\bf v}_h}_v.
\end{equation}
(see \cite[Lemmas~4.2 and~4.3]{Rhebergen:2017}) An immediate
consequence of \cref{eq:stab_bound_ah} is:
\begin{equation}
  \label{eq:norm_equiv_ah}
  c_a^s\tnorm{{\bf v}_h}_v^2 \le a_h({\bf v}_h, {\bf v}_h) \le c_a^b\tnorm{{\bf v}_h}_v^2.
\end{equation}
From \cite[Lemma~4.8 and Eq.~102]{Rhebergen:2017}, there exists a
constant $c_b^b > 0$, independent of $h$, such that for all ${\bf v}_h
\in V_h^{\star}$ and for all ${\bf q}_h \in Q_h^{\star}$
\begin{equation}
  \label{eq:bound_bh}
  \envert{b_h({\bf q}_h, {\bf v}_h)} \le c_b^b \tnorm{{\bf v}_h}_v\tnorm{{\bf q}_h}_p.
\end{equation}

%------------------------------------------------------------------------------
\subsection{The inf-sup condition}
\label{ss:infsupcondition}

We present in this section a proof of inf-sup stability that is
simpler than that in \cite{Rhebergen:2017}, and which better lends
itself to the analysis of preconditioners.

The velocity--pressure coupling term in \cref{eq:discrete_problem} is
\begin{equation}
  \label{eq:formB_b1b2}
  b_h({\bf p}_h, {\bf v}_h) := b_1(p_h, {\bf v}_h) + b_2(\bar{p}_h, {\bf v}_h),
\end{equation}
where
\begin{equation}
  \label{eq:bh1bh2}
  b_1(p_h, {\bf v}_h) := -\sum_{K\in\mathcal{T}} \int_K p_h\nabla\cdot v_h \dif x
  \quad \text{and} \quad
  b_2(\bar{p}_h, {\bf v}_h) := \sum_{K\in\mathcal{T}}\int_{\partial K} v_h\cdot n \bar{p}_h \dif s,
\end{equation}
The main result of this section is stability of~$b_h(\cdot, \cdot):
Q_h^{\star} \times V_h^{\star} \to \mathbb{R}$, which we first state and
then prove after some intermediate results.
\begin{lemma}[Stability of $b_h$]
  \label{thm:stab_bh}
  There exists a constant $\beta_p > 0$, independent of $h$, such that
  for all~${\bf q}_h \in Q_h^{\star}$
  \begin{equation}
    \label{eq:stab_bh}
    \beta_p \tnorm{{\bf q}_h}_{p} \le \sup_{{\bf v}_h\in V_h^{\star}}
    \frac{ b_h({\bf q}_h, {\bf v}_h) }{\tnorm{ {\bf v}_h }_v}.
  \end{equation}
\end{lemma}
Satisfaction of the stability condition does rely on a suitable
combination of function spaces, as chosen in
\cref{eqn:spaces_cell,eqn:spaces_facet}.

The following is a reduced version of~\cite[Theorem~3.1]{Howell:2011}.
\begin{theorem}
  \label{thm:howellWalkington}
  Let $U$, $P_1$, and $P_2$ be reflexive Banach spaces, and let $b_1:
  P_1 \times U \to \mathbb{R}$, and $b_2 : P_2 \times U \to
  \mathbb{R}$ be bilinear and continuous. Let
  \begin{equation}
    Z_{b_i} = \cbr{ v \in U : b_i(p_i, v) = 0 \quad \forall p_i \in P_i} \subset U, \quad i=1,2,
  \end{equation}
  then the following are equivalent:
  \begin{enumerate}
  \item There exists $c > 0$ such that
    \begin{equation*}
      \label{eq:combined_inf_sup}
      \sup_{v\in U} \frac{b_1(p_1, v) + b_2(p_2, v)}{\norm{v}_U} \ge c\del{\norm{p_1}_{P_1} + \norm{p_2}_{P_2}}
      \quad (p_1, p_2) \in P_1 \times P_2.
    \end{equation*}
  \item There exists $c > 0$ such that
    \begin{equation*}
      \label{eq:proofOverZb2U}
      \sup_{v \in Z_{b_2}} \frac{b_1(p_1, v)}{\norm{v}_U} \ge c \norm{p_1}_{P_1},\ p_1\in P_1 \ \mbox{and} \
      \sup_{v \in U} \frac{b_2(p_2, v)}{\norm{v}_U} \ge c \norm{p_2}_{P_2},\ p_2\in P_2.
    \end{equation*}
  \end{enumerate}
\end{theorem}
\Cref{thm:howellWalkington} allows $b_{1}$ and $b_{2}$ in
\cref{eq:formB_b1b2} to be analyzed separately.

\begin{lemma}[Stability of $b_1$]
  \label{thm:stab_b1}
  Let $V_h^{\rm BDM}$ be a Brezzi--Douglas--Marini (BDM) finite
  element space \cite{Boffi:book}:
  \begin{equation}
    % \label{eq:bdm_space}
    \begin{split}
      V_h^{\rm BDM}(K)
      &:=
      \cbr{ v_h \in \sbr{P_k(K)}^d : v_h\cdot n \in L^2(\partial K),\
        v_h\cdot n|_F\in P_k(F)},
      \\
      V_h^{\rm BDM}
      &:=
       \cbr{v_h \in H({\rm div};\Omega) :\ v_h|_K \in V_h^{\rm BDM}(K),\
        \forall K \in \mathcal{T}}.
    \end{split}
  \end{equation}
  There exists a constant $\beta > 0$, independent of $h$, such that
  for all~$q_h \in Q_h$
  \begin{equation}
    \label{eq:stab_b1}
    \beta \norm{q_h}_{0,\Omega} \le \sup_{{\bf v}_h\in V_h^{\star{\rm BDM}}}
    \frac{ b_1(q_h, {\bf v}_h) }{\tnorm{ {\bf v}_h }_v}.
  \end{equation}
\end{lemma}
\begin{proof}
  See~\cite[Lemma~4.4]{Rhebergen:2017}. \qed
\end{proof}

\begin{definition}[BDM lifting operator]
  Let $L^{\rm BDM} : P_k(\partial K) \to \sbr{P_k(K)}^d$ be a BDM
  local lifting of the normal trace defined via the BDM interpolant
  \cite[Example~2.5.1]{Boffi:book} with zero on the interior, which
  has the properties:
  \begin{equation}
    \label{eq:localBDM}
    \del{L^{\rm BDM}\bar{q}_h}\cdot n = \bar{q}_h
      \quad \mbox{and} \quad
    \norm{L^{\rm BDM}\bar{q}_h}_K
    \le
    ch_K^{1/2}\norm{\bar{q}_h}_{\partial K} \qquad
    \forall \bar{q}_h\in P_k(\partial K),
  \end{equation}
  where the inequality follows by a scaling argument.
\end{definition}

It follows then by an inverse estimate that
\begin{equation}
  \label{eq:bound_nabla_BDM}
  \norm{\nabla L^{\rm BDM}\bar{q}_h}_K^2 \le ch_K^{-2}\norm{L^{\rm BDM}\bar{q}_h}_K^2
  \le ch_K^{-1}\norm{\bar{q}_h}^2_{\partial K},
\end{equation}
and by the trace inequality that
\begin{equation}
  \label{eq:bound_boundary_BDM}
  \norm{L^{\rm BDM}\bar{q}_h}_{\partial K}^2
  \le
  ch_K^{-1}\norm{L^{\rm BDM}\bar{q}_h}_{K}^2 \le c\norm{\bar{q}_h}_{\partial K}^2.
\end{equation}
which yields
\begin{equation}
  \norm{\nabla L^{\rm BDM} \bar{q}_h}_{K}^2
  + \frac{\alpha_v}{h_K} \norm{L^{\rm BDM} \bar{q}_h}^2_{\partial K}
  \le
  c h_K^{-1} \norm{\bar{q}_h}_{\partial K}^2.
\end{equation}

\begin{lemma}[Stability of $b_2$]
\label{thm:stab_b2}
  There exists a constant $\bar{\beta} > 0$, independent of $h$, such
  that for all~$\bar{q}_h \in \bar{Q}_h$
  \begin{equation}
    \label{eq:stab_b2}
    \bar{\beta} \norm{\bar{q}_h}_{p} \le \sup_{{\bf v}_h\in V_h^{\star}}
    \frac{ b_2(\bar{q}_h, {\bf v}_h) }{\tnorm{ {\bf v}_h }_v}.
  \end{equation}
\end{lemma}
\begin{proof}
  Summing over all cells and by definition of the norm
  $\tnorm{(\cdot, \cdot)}_v$ in \cref{eq:stability_norm},
  \begin{equation}
    \tnorm{(L^{\rm BDM}\bar{q}_h, 0)}_v
    \le c \sum_{K \in \mathcal{T}} h_K^{-1/2} \norm{\bar{q}_h}_{\partial K}.
  \end{equation}
  Using the above, we have:
  \begin{equation}
    \label{eq:infsuponK}
    \begin{split}
      \sup_{{\bf v}_h\in V_h^{\star}}
      \frac{ \sum_{K\in\mathcal{T}}\int_{\partial K} v_h\cdot n \bar{q}_h \dif s}{\tnorm{{\bf v}_h}_v}
      &\ge
      \frac{ \sum_{K\in\mathcal{T}}\int_{\partial K} \bar{q}_h^2 \dif s}{\tnorm{(L^{\rm BDM}\bar{q}_h,0)}_v}
      \\
      &\ge
      c\frac{\sum_{K\in\mathcal{T}} \norm{\bar{q}_h}^2_{\partial K}}{\sum_{K\in\mathcal{T}}h_K^{-1/2}\norm{\bar{q}_h}_{\partial K}}
      \\
      &\ge c h_{\min}^{1/2} \sum_{K\in\mathcal{T}} \norm{\bar{q}_h}_{\partial K}
      \\
      &\ge c c_{*}^{1/2} \norm{\bar{q}_h}_p,
    \end{split}
  \end{equation}
  where $c_{*} = h_{\min}/h_{\max}$. \qed
\end{proof}

We can now prove the main stability result.
\begin{proof}[Proof of \cref{thm:stab_bh}]
  Using \cref{thm:howellWalkington}, let $b_1(\cdot,\cdot)$ and
  $b_2(\cdot, \cdot)$ be defined as in \cref{eq:bh1bh2}, let $U =
  V_h^{\star}$, $P_1 = Q_h$ and $P_2 = \bar{Q}_h$. Furthermore, note
  that $Z_{b_2} = V_h^{\star{\rm BDM}} \subset V_h^{\star}$. The
  conditions in \cref{eq:proofOverZb2U} were proven in
  \cref{thm:stab_b1,thm:stab_b2}, respectively. By equivalence of
  \cref{eq:combined_inf_sup,eq:proofOverZb2U}, we obtain
  \begin{equation}
    \label{eq:combined_inf_sup_QhbarQh}
      \sup_{{\bf v}_h \in V_h^{\star}} \frac{b_1(q_h, {\bf v}_h) + b_2(\bar{q}_h, {\bf v}_h)}{\tnorm{{\bf v}_h}_v}
      \ge c\del{\norm{q_h}_{0,\Omega} + \norm{\bar{q}_h}_{p}} \quad (q_h, \bar{q}_h) \in Q_h \times \bar{Q}_h,
  \end{equation}
  from which \cref{eq:stab_bh} follows. \qed
\end{proof}

%------------------------------------------------------------------------------
\subsection{Reduced problem}

In practice, a reduced global problem is solved in which $u_{h}$ is
eliminated cell-wise. We present the reduced problem in a variational
setting here for later use in constructing preconditioners.  To
formulate a reduced problem, we first introduce local solvers.
\begin{definition}[Local solver]
  \label{def:localsolver_alt}
  On an element $K$, consider the `local' bilinear and linear forms:
  \begin{equation}
    a_{K}(v_{h}, w_{h}) := \del{\nabla v_h, \nabla w_h}_K
    - \left\langle \partial_n v_h, w_h \right\rangle_{\partial K}
    - \left\langle v_h, \partial_n w_h \right\rangle_{\partial K}
     + \alpha  h_K^{-1} \left\langle v_h,  w_h \right\rangle_{\partial K}
  \end{equation}
  and
  \begin{multline}
    L_{K}(w_{h}) := \del{s, w_h}_K
      - \left\langle \partial_n w_h, \bar{m}_h \right\rangle_{\partial K}
        + \alpha  h_K^{-1} \left\langle w_h, \bar{m}_h \right\rangle_{\partial K}
        \\
        + \del{\nabla\cdot w_h, r_h}_K
        - \left\langle w_h \cdot n, \bar{r}_h \right\rangle_{\partial K}.
  \end{multline}
  The function $v_h^{L}(\bar{m}_h, r_h, \bar{r}_h, s) \in V_{h}$ is such
  that its restriction to element $K$ satisfies the local problem: given
  $s \in \sbr{L^2(\Omega)}^d$ and $(\bar{m}_h, r_h, \bar{r}_h) \in
  \bar{V}_h \times Q_h \times \bar{Q}_h$
  \begin{equation}
    a_{K}\del{v_{h}^{L}, w_{h}} = L_{K}\del{w_{h}} \quad \forall w_{h} \in V(K).
  \label{eqn:local_problem}
  \end{equation}
  where $V(K) := \sbr{P_k(K)}^d$ the polynomial space in which the
  velocity is approximated on a cell.
\end{definition}

We next state the weak formulation of the Stokes problem in which $u_h$
is eliminated from \cref{eq:discrete_problem} by using the local solver
to express the velocity field and the velocity test function on cells,
and phrasing the problem in terms of the pressure trial/test function on
cells and the interface functions.
\begin{lemma}[Weak formulation of the reduced Stokes problem]
  \label{lem:statcon_problem_alt}
  Suppose $({\bf u}_h, {\bf p}_h) \in X_h^{\star}$ satisfy
  \cref{eq:discrete_problem} and $f \in \sbr{L^2(\Omega)}^{d}$. The
  velocity field $u_{h}$ is the sum of the local solutions (from
  \cref{def:localsolver_alt}) $l(\bar{u}_h, {\bf p}_h) :=
  v_{h}^{L}(\bar{u}_h, p_h, \bar{p}_h, 0)$ and $u_{h}^{f} :=
  v_{h}^{L}(0, 0, 0, f)$:
  \begin{equation}
    \label{eq:u_U_p_P_alt}
    u_h = u_h^f + l(\bar{u}_h, {\bf p}_h).
  \end{equation}
  Furthermore, $(\bar{u}_h, {\bf p}_h) \in \bar{V}_h \times Q_h^{\star}$
  satisfies
  \begin{equation}
    \label{eq:global_problem_alt}
    \mathcal{B}_h\del{ \del{\bar{u}_h, {\bf p}_h, }, \del{\bar{w}_h, {\bf q}_h}}
    = \mathcal{L}_h\del{ (\bar{w}_h, {\bf q}_h) } \quad
    \forall \del{\bar{w}_h, {\bf q}_h} \in \bar{V}_h \times Q_h^{\star},
  \end{equation}
  where
  \begin{multline}
    \label{eq:def_B_L_B_alt}
    \mathcal{B}_h\del{ \del{\bar{v}_h, {\bf r}_h}, \del{\bar{w}_h, {\bf q}_h}}
    :=
    a_h\del{\del{ l(\bar{v}_{h}, \vec{r}_{h}), \bar{v}_h}, \del{l(\bar{w}_{h}, \vec{q}_{h}), \bar{w}_h}}
    \\
    + b_h\del{\vec{r}_h, \del{l(\bar{w}_{h}, \vec{q}_{h}), \bar{w}_h}}
    + b_h\del{\vec{q}_h, \del{l(\bar{v}_{h}, \vec{r}_{h}), \bar{v}_h}}
\end{multline}
  and
  \begin{equation}
    \label{eq:def_B_L_L_alt}
    \mathcal{L}_h( (\bar{w}_h, {\bf q}_h) ) := (l(\bar{w}_{h}, \vec{q}_{h}), f)_{\mathcal{T}},
  \end{equation}
  where $l(\bar{v}_{h}, \vec{r}_{h}) := l(\bar{v}_{h}) +
  l(\vec{r}_{h})$, and $l(\bar{v}_{h}) := v_{h}^{L}(\bar{v}_{h}, 0, 0,
  0)$ and $l(\vec{r}_{h}) := v_{h}^{L}(0, r_{h}, \bar{r}_{h}, 0)$ (by
  \cref{def:localsolver_alt}).
\end{lemma}
\begin{proof}
  \Cref{eq:u_U_p_P_alt} follows by \cref{eq:discrete_problem_a},
  linearity of the problem and \cref{def:localsolver_alt} (local
  solver).

  We next prove \cref{eq:global_problem_alt}. Note that $l(\bar{w}_{h},
  \vec{q}_{h})$, restricted to the cell $K$, satisfies a form of the
  local problem in \cref{eqn:local_problem}, and $u_h^f$, restricted to
  the cell $K$, satisfies a form of the local problem. Combining the two
  local problems (the former with $u_{h}^{f}$ in the test function slot,
  and the latter with $l(\bar{w}_{h}, \vec{q}_{h})$ in the test function
  slot), and summing over the cells in the triangulation,
  \begin{multline}
    \label{eq:simplified_rhs}
    - \left\langle \partial_n u_h^f, \bar{w}_h \right\rangle_{\partial \mathcal{T}}
    + \alpha h_K^{-1} \left\langle u_h^f, \bar{w}_h \right\rangle_{\partial \mathcal{T}}
\\
    =
    \del{f, l(\bar{w}_{h}, \vec{q}_{h})}_{\mathcal{T}}
    - \del{\nabla\cdot u_h^f, q_h}_{\mathcal{T}}
    + \left \langle u_h^f \cdot n, \bar{q}_h \right\rangle_{\partial \mathcal{T}}.
  \end{multline}
  The lifted function $l(\bar{u}_h, {\bf p}_h)$, restricted to the cell
  $K$, satisfies the local problem, which with $l(\bar{w}_{h},
  \vec{q}_{h})$ in the test function slot reads:
  \begin{multline}
    \label{eq:simp_uh0}
    \del{\nabla l(\bar{u}_h, {\bf p}_h), \nabla l(\bar{w}_{h}, \vec{q}_{h})}_K
    - \left\langle \partial_n l(\bar{u}_h, {\bf p}_h), l(\bar{w}_{h}, \vec{q}_{h}) \right\rangle_{\partial K}
    \\
    - \left\langle \partial_n l(\bar{w}_{h}, \vec{q}_{h}), l(\bar{u}_h, {\bf p}_h) - \bar{u}_h \right\rangle_{\partial K}
    + \alpha h_K^{-1} \left\langle l(\bar{u}_h, {\bf p}_h) - \bar{u}_h,  l(\bar{w}_{h}, \vec{q}_{h}) \right\rangle_{\partial K}
\\
    - \del{\nabla \cdot l(\bar{w}_{h}, \vec{q}_{h}), p_h}_K
    + \left \langle l(\bar{w}_{h}, \vec{q}_{h}) \cdot n, \bar{p}_h \right\rangle_{\partial K} = 0.
  \end{multline}
  Substituting \cref{eq:u_U_p_P_alt} into \cref{eq:discrete_problem},
  with ${\bf v}_h = (0, \bar{w}_h)$ in \cref{eq:discrete_problem_a},
  \begin{subequations}
    \begin{align}
      \label{eq:fluxes_uf}
      \left\langle \partial_n l(\bar{u}_h, {\bf p}_h), \bar{w}_h \right\rangle_{\partial\mathcal{T}}
      - \alpha h_K^{-1} \left\langle l(\bar{u}_h, {\bf p}_h) - \bar{u}_h, \bar{w}_h \right\rangle_{\partial\mathcal{T}}
      &=
        - \left\langle \partial_n u_h^f, \bar{w}_h \right\rangle_{\partial\mathcal{T}}
        + \alpha  h_K^{-1} \left\langle u^f_h, \bar{w}_h \right\rangle_{\partial\mathcal{T}},
      \\
      \label{eq:divergences_uf}
      \del{q_h, \nabla \cdot l(\bar{u}_h, {\bf p}_h)}_{\mathcal{T}}
      - \left\langle l(\bar{u}_h, {\bf p}_h) \cdot n, \bar{q}_h \right\rangle_{\partial \mathcal{T}}
      &=
        - \del{q_h, \nabla \cdot u_h^f}_{\mathcal{T}}
        + \left\langle u_h^f \cdot n, \bar{q}_h \right\rangle_{\partial \mathcal{T}}.
    \end{align}
  \end{subequations}
  Summing \cref{eq:simp_uh0} over all cells and adding to the left-hand
  side of \cref{eq:fluxes_uf}, and using \cref{eq:simplified_rhs} to
  replace the right-hand side of \cref{eq:fluxes_uf}, we have:
  \begin{multline}
    \del{\nabla l(\bar{u}_h, {\bf p}_h), \nabla l(\bar{w}_{h}, \vec{q}_{h})}_{\mathcal{T}}
    - \left\langle \partial_n l(\bar{u}_h, {\bf p}_h), l(\bar{w}_{h}, \vec{q}_{h}) - \bar{w}_h \right\rangle_{\partial \mathcal{T}}
    \\
    - \left\langle \partial_n l(\bar{w}_{h}, \vec{q}_{h}), l(\bar{u}_h, {\bf p}_h) - \bar{u}_h \right\rangle_{\partial \mathcal{T}}
    \\
    + \alpha h_K^{-1} \left\langle l(\bar{u}_h, {\bf p}_h) - \bar{u}_h,  l(\bar{w}_{h}, \vec{q}_{h}) - \bar{w}_h\right\rangle_{\partial \mathcal{T}}
    - \del{\nabla \cdot l(\bar{w}_{h}, \vec{q}_{h}), p_h}_{\mathcal{T}}
    \\
    + \left \langle l(\bar{w}_{h}, \vec{q}_{h}) \cdot n, \bar{p}_h \right\rangle_{\partial \mathcal{T}}
    + \del{\nabla\cdot u_h^f, q_h}_{\mathcal{T}}
    - \left \langle u_h^f \cdot n, \bar{q}_h \right\rangle_{\partial \mathcal{T}}
    =
    \del{f, l(\bar{w}_{h}, \vec{q}_{h})}_{\mathcal{T}}.
  \end{multline}
  \Cref{eq:global_problem_alt} follows after using
  \cref{eq:divergences_uf}. \qed
\end{proof}

By \cref{def:localsolver_alt}, consider $l(\bar{v}_{h}) :=
v_{h}^{L}(\bar{v}_h, 0, 0, 0)$ and $l(\vec{r}_{h}) := v_{h}^{L}(0,
r_h, \bar{r}_h, 0)$. By linearity, it follows that $l(\bar{v}_{h},
\vec{r}_{h}) = l(\Bar{v}_{h}) + l(\vec{r}_{h})$. Note also
\begin{multline}
  \label{eq:a_in_bits}
  a_h((l(\bar{v}_{h}, \vec{r}_{h}), \bar{v}_h), (l(\bar{w}_{h}, \vec{q}_{h}), \bar{w}_h))
  =
  \underbrace{a_h((l(\bar{v}_h), \bar{v}_h), (l(\bar{w}_h),
    \bar{w}_h))}_{\bar{a}_{h}(\bar{v}_{h}, \bar{w}_{h})}
  \\
  +
  a_h((l(\bar{v}_h), \bar{v}_h), (l(\vec{q}_h), \bar{w}_h))
  +
  a_h((l(\vec{r}_{h}), \bar{v}_h), (l(\bar{w}_{h}, \vec{q}_{h}), \bar{w}_h)),
\end{multline}
where $l(\bar{w}_{h}) := v_{h}^{L}(\bar{w}_h, 0, 0, 0)$,
$l(\vec{q}_{h}) := v_{h}^{L}(0, q_h, \bar{q}_h, 0)$ and
$l(\bar{w}_{h}, \vec{q}_{h}) := l(\bar{w}_{h}) + l(\vec{q}_{h})$. The
following result for $\bar{a}_h(\cdot, \cdot)$ will be useful in
analyzing preconditioners.
\begin{lemma}[Equivalence of norm induced by $\bar{a}_h(\cdot, \cdot)$]
  \label{lem:specequiv_bara}
  There exist positive constants $C_1$ and $C_2$ independent of $h_K$
  such that
  \begin{equation}
    \label{eq:specequiv_bara}
    C_1 \tnorm{\bar{w}_h}_h^2
    \le
    \bar{a}_h(\bar{w}_h, \bar{w}_h)
    \le
    C_2 \tnorm{\bar{w}_h}_h^2 \quad \forall \bar{w}_{h} \in \bar{V}_{h}.
  \end{equation}
\end{lemma}
\begin{proof}
  From boundedness and coercivity of $a_h(\cdot, \cdot)$ (see
  \cref{eq:norm_equiv_ah}), for sufficiently large~$\alpha$,
  \begin{equation}
    c_a^s\tnorm{(l(\bar{w}_{h}), \bar{w}_h)}_v^2
    \le \bar{a}_h(\bar{w}_h, \bar{w}_h)
    \le c_a^b \tnorm{(l(\bar{w}_{h}), \bar{w}_h)}_v^2.
  \end{equation}
  We therefore need to demonstrate the equivalence
  \begin{equation}
    \label{eq:toprove}
    c_1 \tnorm{\bar{w}_h}_h \le \tnorm{(l(\bar{w}_{h}), \bar{w}_h)}_v
    \le c_2 \tnorm{\bar{w}_h}_h.
  \end{equation}

  We first consider the lower bound in \cref{eq:toprove}. Note that
  \begin{equation}
    \label{eq:splittingnorm}
    \begin{split}
      h_K^{-1/2}\norm{\bar{w}_h - m_K(\bar{w}_h)}_{\partial K}
      &\le
      h_K^{-1/2}\norm{\bar{w}_h - l(\bar{w}_{h}) }_{\partial K}
      + h_K^{-1/2} \norm{l(\bar{w}_{h}) - m_K(\bar{w}_h)}_{\partial K}.
    \end{split}
  \end{equation}
  Defining
  \begin{equation}
    M_K(l(\bar{w}_{h})) := \frac{1}{|K|}\int_K l(\bar{w}_{h}) \dif s,
  \end{equation}
  then
  \begin{equation}
    \label{eq:w_avg}
    \begin{split}
      &h_K^{-1/2}\norm{l(\bar{w}_{h}) - m_K(\bar{w}_h)}_{\partial K}
      \\
      &\quad\le
      h_K^{-1/2}\norm{l(\bar{w}_{h}) - M_K(l(\bar{w}_{h}))}_{\partial K}
      + h_K^{-1/2}\norm{M_K(l(\bar{w}_{h})) - m_K(\bar{w}_h)}_{\partial K}
      \\
      &\quad=
      h_K^{-1/2}\norm{l(\bar{w}_{h}) - M_K(l(\bar{w}_{h}))}_{\partial K}
      + h_K^{-1/2}\norm{m_K(M_K(l(\bar{w}_{h})) - \bar{w}_h)}_{\partial K}
      \\
      &\quad\le
      h_K^{-1/2}\norm{l(\bar{w}_{h}) - M_K(l(\bar{w}_{h}))}_{\partial K}
      + h_K^{-1/2}\norm{M_K(l(\bar{w}_{h})) - \bar{w}_h}_{\partial K}
      \\
      &\quad\le
      2h_K^{-1/2}\norm{l(\bar{w}_{h}) - M_K(l(\bar{w}_{h}))}_{\partial K}
      + h_K^{-1/2}\norm{l(\bar{w}_{h}) - \bar{w}_h}_{\partial K},
    \end{split}
  \end{equation}
  where the second inequality is by~\cite[Eq.~(10.6.11)]{Brenner:book}.
  By a trace inequality and a (scaled) Friedrich's
  inequality~\cite[Lemma 4.3.14]{Brenner:book},
  \begin{equation}
    \label{eq:trace_fried}
    2h_K^{-1/2}\norm{l(\bar{w}_{h}) - M_K(l(\bar{w}_{h}))}_{\partial K}
    \le
    2h_K^{-1}\norm{l(\bar{w}_{h}) - M_{K}(l(\bar{w}_{h}))}
    \le
    c \norm{\nabla l(\bar{w}_{h})}_K.
  \end{equation}
  Combining \cref{eq:splittingnorm,eq:w_avg,eq:trace_fried},
  \begin{equation}
    h_K^{-1/2}\norm{\bar{w}_h - m_K(\bar{w}_h)}_{\partial K}
    \le
    c \del{ \norm{\nabla l(\bar{w}_{h})}_K + h_K^{-1/2} \norm{l(\bar{w}_{h})
    - \bar{w}_h}_{\partial K}}.
  \end{equation}
  The lower bound in \cref{eq:toprove} follows after squaring, applying
  Young's inequality, and summing over all cells.

  We next consider the upper bound in \cref{eq:toprove}. By definition
  of $l(\bar{w}_{h})$ and considering that $m_{K}(\bar{w}_{h})$ is
  constant on a cell,
  \begin{multline}
    \label{eq:vhuequation2}
    \del{\nabla l(\bar{w}_{h}), \nabla w_h}_K
    - \left\langle \partial_n l(\bar{w}_{h}), w_h \right\rangle_{\partial K}
    - \left\langle l(\bar{w}_{h}) - m_K(\bar{w}_h), \partial_n w_h \right\rangle_{\partial K}
\\
    + \alpha  h_K^{-1} \left\langle l(\bar{w}_{h}) - m_K(\bar{w}_h),  w_h \right\rangle_{\partial K}
    =
    - \left\langle \partial_n w_h, \bar{w}_h - m_K(\bar{w}_h) \right\rangle_{\partial K}
  \\
    + \alpha  h_K^{-1} \left\langle w_h, \bar{w}_h - m_K(\bar{w}_h) \right\rangle_{\partial K}
    \quad  \forall w_h \in V(K),
  \end{multline}
  Setting $w_h = l(\bar{w}_{h}) - m_K(\bar{w}_h)$, then
  \begin{multline}
    \label{eq:vhuequation3}
    \norm{\nabla l(\bar{w}_{h})}^2_K
    - 2\left\langle \partial_n l(\bar{w}_{h}), l(\bar{w}_{h}) - m_K(\bar{w}_h) \right\rangle_{\partial K}
    \\
    + \alpha  h_K^{-1} \left\langle l(\bar{w}_{h}) - m_K(\bar{w}_h),  l(\bar{w}_{h}) - m_K(\bar{w}_h) \right\rangle_{\partial K}
    \\
    =
    - \left\langle \partial_n l(\bar{w}_{h}), \bar{w}_h - m_K(\bar{w}_h) \right\rangle_{\partial K}
   \\
    + \alpha  h_K^{-1} \left\langle l(\bar{w}_{h}) - m_K(\bar{w}_h), \bar{w}_h - m_K(\bar{w}_h) \right\rangle_{\partial K},
  \end{multline}
  which can be manipulated into
  \begin{multline}
    \label{eq:vhuequation4}
    \norm{\nabla l(\bar{w}_{h})}^2_K
    + 2\left\langle \partial_n l(\bar{w}_{h}), \bar{w}_h - w_h^{\bar{w}}\right\rangle_{\partial K}
    + \alpha  h_K^{-1} \norm{l(\bar{w}_{h}) - \bar{w}_h}^2_{\partial K}
    \\
    =
    \left\langle \partial_n w_h^{\bar{w}}, \bar{w}_h - m_K(\bar{w}_h) \right\rangle_{\partial K}
    + \alpha h_K^{-1}\left\langle \bar{w}_h - m_K(\bar{w}_h), \bar{w}_h - l(\bar{w}_{h})\right\rangle_{\partial K}.
  \end{multline}
  Considering the left-hand side of \cref{eq:vhuequation4}, by
  coercivity of~$a_h(\cdot, \cdot)$,
  \begin{multline}
    \label{eq:normingvhumequiv}
    c_1\del{\norm{\nabla l(\bar{w}_{h})}_K^2
      + \alpha  h_K^{-1} \norm{l(\bar{w}_{h}) - \bar{w}_h}^2_{\partial K}}
    \\ \le
    \norm{\nabla l(\bar{w}_{h})}_K^2
    + 2\left\langle \partial_n l(\bar{w}_{h}), \bar{w}_h - l(\bar{w}_{h}) \right\rangle_{\partial K}
    + \alpha  h_K^{-1} \norm{l(\bar{w}_{h}) - \bar{w}_h}^2_{\partial K}.
  \end{multline}
  For the right-hand side of \cref{eq:vhuequation4}, by Cauchy--Schwarz
  and a trace inequality,
  \begin{equation}
    \label{eq:vhumrhsbound}
    \begin{split}
      &\left\langle \partial_n l(\bar{w}_{h}), \bar{w}_h - m_K(\bar{w}_h)\right\rangle_{\partial K}
      + \alpha h_K^{-1}\left\langle \bar{w}_h - m_K(\bar{w}_h), \bar{w}_h - l(\bar{w}_{h})\right\rangle_{\partial K}
      \\
      &\le
      h_K^{1/2}\norm{ \partial_n l(\bar{w}_{h})}_{\partial K} h_K^{-1/2} \norm{\bar{w}_h - m_K(\bar{w}_h)}_{\partial K}
      \\
       & \quad + \alpha^{1/2} h_K^{-1/2}\norm{\bar{w}_h - m_K(\bar{w}_h)}_{\partial K} \alpha^{1/2} h_K^{-1/2}\norm{l(\bar{w}_{h}) - \bar{w}_h}_{\partial K}
      \\
      & \le
      \norm{ \nabla l(\bar{w}_{h}) }_{ K} h_K^{-1/2} \norm{\bar{w}_h - m_K(\bar{w}_h)}_{\partial K}
      \\
      & \quad + \alpha^{1/2}h_K^{-1/2}\norm{\bar{w}_h - m_K(\bar{w}_h)}_{\partial K} \alpha^{1/2} h_K^{-1/2}\norm{l(\bar{w}_{h}) - \bar{w}_h}_{\partial K}
      \\
      & \le
      \del{\norm{ \nabla l(\bar{w}_{h}) }_{ K} +  \alpha^{1/2} h_K^{-1/2}\norm{l(\bar{w}_{h}) - \bar{w}_h}_{\partial K}}
      \del{\alpha^{1/2}h_K^{-1/2}\norm{\bar{w}_h - m_K(\bar{w}_h)}_{\partial K}}.
    \end{split}
  \end{equation}
  Combining \cref{eq:vhuequation4,eq:vhumrhsbound,eq:normingvhumequiv},
  \begin{equation}
    \norm{\nabla l(\bar{w}_{h})}_K
    + \alpha^{1/2} h_K^{-1/2} \norm{l(\bar{w}_{h}) - \bar{w}_h}_{\partial K}
    \le
    c \alpha^{1/2} h_K^{-1/2}\norm{\bar{w}_h - m_K(\bar{w}_h)}_{\partial K}.
  \end{equation}
  The upper bound in \cref{eq:toprove} follows after squaring,
  application of Young's inequality, and summing over all elements.
  \qed
\end{proof}

%------------------------------------------------------------------------------
\section{Preconditioning}
\label{s:preconditioning}

We present now the main results, namely preconditioners for the
hybridized discontinuous Galerkin discretization of the Stokes
equations. We consider first a preconditioner for the full problem (no
static condensation), and then the system with the cell-wise velocity
degrees-of-freedom eliminated locally. As mentioned in the
introduction, we do not consider the case where the cell-wise pressure
degrees-of-freedom are also eliminated locally as this complicates
preconditioning and the reduction in size of the global systems
through eliminating the pressure is modest. Preconditioning a system
with both the velocity and pressure degrees-of-freedom eliminated
cell-wise is an interesting technical question.

%------------------------------------------------------------------------------
\subsection{The full discrete Stokes problem}
\label{ss:discretesystem}

Let $u \in \mathbb{R}^{n_u}$ be the vector of discrete velocity with
respect to the basis for $V_h$, and $p \in N^{n_p} = \{q \in
\mathbb{R}^{n_p} | q \ne 1\}$ be the vector of the discrete pressure
with respect to the basis for $Q_h$. Furthermore, let $\bar{u} \in
\mathbb{R}^{\bar{n}_u}$ and $\bar{p} \in \mathbb{R}^{\bar{n}_p}$ be the
vectors of discrete velocity and pressure associated with the spaces
$\bar{V}_h$ and $\bar{Q}_h$, respectively. The discrete problem in
\cref{eq:discrete_problem} can be expressed as the system of linear
equations:
\begin{equation}
  \label{eq:block-matrix}
  \begin{bmatrix}
    A & B^T \\
    B & 0
  \end{bmatrix}
  \begin{bmatrix}
    U \\
    P
  \end{bmatrix}
  =
  \begin{bmatrix}
    L \\ 0
  \end{bmatrix},
  \quad
  \text{with}
  \quad
  U :=
  \begin{bmatrix}
    u \\ \bar{u}
  \end{bmatrix},
  \quad
  P :=
  \begin{bmatrix}
    p \\ \bar{p}
  \end{bmatrix},
  \quad
  L :=
  \begin{bmatrix}
    L_u \\ L_{\bar{u}}
  \end{bmatrix},
\end{equation}
and where $A$ and $B$ are the matrices obtained from the
discretization of the bilinear forms $a_h(\cdot, \cdot)$ and
$b_h(\cdot, \cdot)$, defined by \cref{eq:bilin_forms}. The matrices
$A$ and $B$ are block matrices:
\begin{equation}
  \label{eq:blocks}
  A :=
  \begin{bmatrix}
    A_{uu} & A_{\bar{u}u}^T \\
    A_{\bar{u}u} & A_{\bar{u}\bar{u}}
  \end{bmatrix}
  \qquad \mbox{and} \qquad
  B :=
  \begin{bmatrix}
    B_{pu} & 0 \\
    B_{\bar{p}u} & 0
  \end{bmatrix},
\end{equation}
where $A_{uu}$, $A_{\bar{u}u}$ and $A_{\bar{u} \bar{u}}$ are the
matrices obtained from the discretization of $a_h((\cdot, 0), (\cdot,
0))$, $a_h((\cdot, 0), (0, \cdot))$ and $a_h((0, \cdot), (0, \cdot))$,
respectively, and $B_{pu}$ and $B_{\bar{p}u}$ are the matrices obtained
from the discretization of $b_h((\cdot,0), (\cdot,0))$ and $b_h((0,
\cdot),(\cdot, 0))$, respectively.

We also introduce `cell'  and `facet' pressure mass matrices, $M$ and
$\bar{M}$, which are obtained from the discretization of
\begin{equation}
  \label{eq:pressureinnerproducts_alt}
  \del{q_h, p_h}_{\mathcal{T}}
    :=
  \sum_{K\in\mathcal{T}}\int_{K}q_hp_h \dif x
  \quad \mbox{and} \quad
  \langle \bar{q}_h, \bar{p}_h \rangle_p
  :=
  \sum_{K\in\mathcal{T}}h_K\int_{\partial K}\bar{q}_h\bar{p}_h \dif s,
\end{equation}
respectively, and note that
\begin{equation}
  \label{eq:def_M_barM}
  \norm{q_h}^2_{\Omega} = q^TMq, \qquad
  \norm{\bar{q}_h}^2_p = \bar{q}^T\bar{M}\bar{q}.
\end{equation}
Defining $\mathcal{M} := {\rm bdiag}(M, \bar{M})$ and $Q := [q^T \
\bar{q}^T]^T$,
\begin{equation}
  \label{eq:def_mathcal_M}
  \tnorm{{\bf q}_h}_p^2 = Q^T\mathcal{M}Q
  = q^TMq + \bar{q}^T\bar{M}\bar{q}.
\end{equation}

\begin{lemma}[Spectral equivalence between the mass matrix and the
  Schur complement]
  \label{lem:spec_equiv_Schur}
  Let $A$ and $B$ be the matrices given in \cref{eq:blocks} and let
  $\mathcal{M}$ be defined as in
  \cref{eq:def_mathcal_M}. Let $\beta_p$ and $c_b^b$ be the constants
  given in \cref{thm:stab_bh,eq:bound_bh}, respectively, and let $c_a^b$
  and $c_a^s$ be the constants given in
  \cref{eq:norm_equiv_ah}. The following holds:
  \begin{equation}
    \label{eq:spec_equiv_mathcalM_Schur}
    \frac{\beta_p}{\sqrt{c_a^b}} \le
    \frac{ Q^T BA^{-1}B^T Q}{Q^T \mathcal{M} Q} \le \frac{c_b^b}{\sqrt{c_a^s}}.
  \end{equation}
\end{lemma}
\begin{proof}
  Stability of $b_h$ (see \cref{thm:stab_bh}) and equivalence of $a_h$
  with $\tnorm{\cdot}_v$ in \cref{eq:norm_equiv_ah} imply
  \begin{equation}
    \label{eq:stab_bh_ah}
    \frac{\beta_p}{\sqrt{c_a^b}} \le \sup_{{\bf v}_h \in V_h^{\star}}
    \frac{ b_h({\bf q}_h, {\bf v}_h) }{a_h({\bf v}_h, {\bf v}_h)^{1/2}\tnorm{{\bf q}_h}_{p}}.
  \end{equation}
  Letting $V = \sbr{v^T \ \bar{v}^T}^T$, with $v \in \mathbb{R}^{n_u}$ and
  $\bar{v} \in \mathbb{R}^{\bar{n}_u}$ and $Q = \sbr{q^T \ \bar{q}^T}^T$,
  with $q \in N^{n_p}$ and $\bar{q} \in \mathbb{R}^{\bar{n}_q}$, we
  can express \cref{eq:stab_bh_ah} in matrix form:
  \begin{equation}
    \label{eq:stab_bh_ah_matrix}
    \frac{\beta_p}{\sqrt{c_a^b}} \le \min_{Q}\max_{V \ne 0}
    \frac{ Q^TBV }
    {\left(V^TAV\right)^{1/2}\left(Q^T\mathcal{M}Q\right)^{1/2}},
  \end{equation}
  which is equivalent to
  \begin{equation}
    \label{eq:stab_bh_mm}
    \frac{\beta_p}{\sqrt{c_a^b}} \le \min_{Q}
    \frac{ Q^TBA^{-1}B^TQ }{ Q^T\mathcal{M}Q },
  \end{equation}
  (see \citep[Section~3]{Pestana:2015}), proving the lower bound in
  \cref{eq:spec_equiv_mathcalM_Schur}.

  For the upper bound, from \cref{eq:bound_bh,eq:norm_equiv_ah} note
  that:
  \begin{equation}
    \label{eq:upperboundspecequivBVQ}
    \envert{b_h({\bf q}_h, {\bf v}_h)}
    \le c_b^b \tnorm{{\bf v}_h}_v\tnorm{{\bf q}_h}_p
    \le \frac{c_b^b}{\sqrt{c_a^s}} a_h({\bf v}_h, {\bf v}_h)^{1/2} \tnorm{{\bf q}_h}_p.
  \end{equation}
  The result follows after dividing both sides of
  \cref{eq:upperboundspecequivBVQ} by $\bar{a}_h({\bf v}_h, {\bf
  v}_h)^{1/2} \tnorm{{\bf q}_h}_{p}$ and expressing in matrix form. \qed
\end{proof}

\Cref{lem:spec_equiv_Schur} can be used to formulate a preconditioner
for the discrete problem in \cref{eq:block-matrix}.
\begin{lemma}[An optimal preconditioner for the full discrete Stokes
problem]
  \label{thm:global_precon}
  Let $A$ and $B$ be the matrices given in
  \cref{eq:blocks} and let $\mathcal{M}$ be defined as in
  \cref{eq:def_mathcal_M}. Let $R$ be an operator that is spectrally
  equivalent to $A$. For the preconditioned system
  \begin{equation}
    \label{eq:precon-block-matrix}
    \mathbb{P}^{-1}\mathbb{A}\mathbb{U} = \mathbb{P}^{-1}\mathbb{F}
    \qquad \leftrightarrow \qquad
    \begin{bmatrix}
      R & 0 \\ 0 & \mathcal{M}
    \end{bmatrix}^{-1}
    \begin{bmatrix}
      A & B^T \\
      B & 0
    \end{bmatrix}
    \begin{bmatrix}
      U \\
      P
    \end{bmatrix}
    =
    \begin{bmatrix}
      R & 0 \\ 0 & \mathcal{M}
    \end{bmatrix}^{-1}
    \begin{bmatrix}
      L \\ 0
    \end{bmatrix},
  \end{equation}
  there exist positive constants $C_1, C_2, C_3, C_4$, independent of
  $h$, such that the negative and positive eigenvalues of
  $\mathbb{P}^{-1}\mathbb{A}$ satisfy $\lambda \in [-C_1,
  -C_2]$ and $\lambda\in[C_3, C_4]$, respectively.
\end{lemma}
\begin{proof}
  By \cref{lem:spec_equiv_Schur}, $\mathcal{M}$ is spectrally equivalent
  to the negative Schur complement $BA^{-1}B^T$ of $\mathbb{A}$.
  Furthermore, since $B$ is full rank (by the discrete inf-sup condition
  \cref{thm:stab_bh}) and since $A$ is symmetric positive definite, the
  result follows by direct application
  of~\citep[Theorem~5.2]{Pestana:2015}. \qed
\end{proof}

%------------------------------------------------------------------------------
\subsection{Preconditioners for the statically condensed Stokes problem}
\label{ss:discrete_sc_system}

The `full' system considered in \cref{ss:discretesystem} is not the
system we wish to solve in practice. We wish to eliminate locally the
cell-wise velocity degrees-of-freedom via static condensation and
precondition the resulting reduced system. We now consider the
elimination of $u$ in \cref{eq:precon-block-matrix} to obtain a linear
system only for $\bar{u}$, $p$ and~$\bar{p}$.

Separating the degrees-of-freedom associated with $V_h$ from those
associated with the Lagrange multipliers, $\bar{V}_h \times
Q_h^{\star}$, we write \cref{eq:block-matrix} as
\begin{equation}
  \label{eq:reorg-discsys}
  \begin{bmatrix}
    A_{uu} & \mathsf{B}^T \\
    \mathsf{B} & \mathsf{C}
  \end{bmatrix}
  \begin{bmatrix}
    u \\ \mathsf{U}
  \end{bmatrix}
  =
  \begin{bmatrix}
    L_u \\ \mathsf{L}
  \end{bmatrix},
  \quad \mbox{with} \quad
  \mathsf{U} :=
  \begin{bmatrix}
    \bar{u} \\ p \\ \bar{p}
  \end{bmatrix} \quad
  \mathsf{L} :=
  \begin{bmatrix}
    L_{\bar{u}} \\ 0 \\ 0
  \end{bmatrix},
\end{equation}
and where
\begin{equation}
  \label{eq:def_sfA_sfB}
  \mathsf{B} :=
  \begin{bmatrix}
    A_{\bar{u}u} \\ B_{pu} \\ B_{\bar{p}u}
  \end{bmatrix},
  \quad
  \mathsf{C} :=
  \begin{bmatrix}
    A_{\bar{u}\bar{u}} & 0 & 0 \\ 0 & 0 & 0 \\ 0 & 0 & 0
  \end{bmatrix}.
\end{equation}
Note that $A_{uu}$ is a block diagonal matrix (one block per cell).
Using $u = A_{uu}^{-1} \del{L_u - \mathsf{B}^T \mathsf{U}}$, we
eliminate $u$ from \cref{eq:reorg-discsys}. This results in a reduced
system for $\mathsf{U}$ only,
\begin{equation}
  \label{eq:system_S_comp}
  \begin{bmatrix}
    \bar{A} & \bar{B}^T \\
    \bar{B} & \bar{C}
  \end{bmatrix}
  \begin{bmatrix}
    \bar{u} \\ P
  \end{bmatrix}
  =
  \begin{bmatrix}
    \bar{L} \\ \bar{G}
  \end{bmatrix},
\end{equation}
where $\bar{A} = -A_{\bar{u}u}A_{uu}^{-1}A_{\bar{u}u}^T + A_{\bar{u}\bar{u}}$ and where
\begin{equation}
  \label{eq:system_S_matrices}
  \bar{B} =
  \begin{bmatrix}
    -B_{pu}A_{uu}^{-1}A_{\bar{u}u}^T \\
    -B_{\bar{p}u}A_{uu}^{-1}A_{\bar{u}u}^T
  \end{bmatrix},
  \quad
  \bar{C} =
  \begin{bmatrix}
    -B_{pu}A_{uu}^{-1}B_{pu}^T & \ -B_{pu}A_{uu}^{-1}B_{\bar{p}u}^T \\
    -B_{\bar{p}u}A_{uu}^{-1}B_{pu}^T & \ -B_{\bar{p}u}A_{uu}^{-1}B_{\bar{p}u}^T
  \end{bmatrix}.
\end{equation}
The Schur complement of the block matrix in~\cref{eq:system_S_comp} is
given by $\bar{S} = -\bar{B} \bar{A}^{-1} \bar{B}^T + \bar{C}$. It is
easy to show (by direct computation) that $\bar{S} = -BA^{-1} B^T$,
with $A$ and $B$ the matrices given in \cref{eq:blocks}. An immediate
consequence of \cref{lem:spec_equiv_Schur} therefore is the following
corollary.

\begin{corollary}[Spectral equivalence between the mass matrix and the
  Schur complement of the statically condensed linear system]
  \label{cor:spec_equiv_sc_Schur}
  Let $\bar{A}$, $\bar{B}$ and $\bar{C}$ be the matrices given in
  \cref{eq:system_S_comp} and let $\mathcal{M}$ be defined as in
  \cref{eq:def_mathcal_M}. Let $\beta_p$ and $c_b^b$ be the constants
  given in, \cref{thm:stab_bh,eq:bound_bh}, respectively, and let
  $c_a^b$ and $c_a^s$ be the constants given in
  \cref{eq:norm_equiv_ah}. The following holds:
  \begin{equation}
    \label{eq:spec_equiv_mathcalM_sc_Schur}
    \frac{\beta_p}{\sqrt{c_a^b}} \le
    \frac{ Q^T \del{\bar{B}\bar{A}^{-1}\bar{B}^T - \bar{C}} Q}{Q^T \mathcal{M} Q}
    \le \frac{c_b^b}{\sqrt{c_a^s}}.
  \end{equation}
\end{corollary}

This corollary can now be used to develop a preconditioner for the
statically condensed linear system in \cref{eq:system_S_comp}.
\begin{theorem}[An optimal preconditioner for the statically condensed
  discrete Stokes problem based on mass matrices]
  \label{thm:global_sc_precon}
  Let $\bar{A}$ and $\bar{B}$ be the matrices given in
  \cref{eq:system_S_comp} and let $\mathcal{M}$ be defined as in
  \cref{eq:def_mathcal_M}. Let $\bar{R}$ be an operator that is
  spectrally equivalent to $\bar{A}$. Consider the preconditioned system
  \begin{equation}
    \label{eq:precon-block-matrix-sc}
    \bar{\mathbb{P}}_{\mathcal{M}}^{-1}\bar{\mathbb{A}}\bar{\mathbb{U}} = \bar{\mathbb{P}}_{\mathcal{M}}^{-1}\bar{\mathbb{F}}
    \qquad \leftrightarrow \qquad
    \begin{bmatrix}
      \bar{R} & 0 \\ 0 & \mathcal{M}
    \end{bmatrix}^{-1}
    \begin{bmatrix}
      \bar{A} & \bar{B}^T \\
      \bar{B} & \bar{C}
    \end{bmatrix}
    \begin{bmatrix}
      \bar{u} \\
      P
    \end{bmatrix}
    =
    \begin{bmatrix}
      \bar{R} & 0 \\ 0 & \mathcal{M}
    \end{bmatrix}^{-1}
    \begin{bmatrix}
      \bar{L} \\ \bar{G}
    \end{bmatrix}.
  \end{equation}
  There exist positive constants $C_1, C_2, C_3, C_4$, independent of
  $h$, such that the negative and positive eigenvalues of
  $\bar{\mathbb{P}}_{\mathcal{M}}^{-1} \bar{\mathbb{A}}$ satisfy
  $\lambda \in [-C_1, -C_2]$ and $\lambda \in [C_3, C_4]$, respectively.
\end{theorem}
\begin{proof}
  By \cref{cor:spec_equiv_sc_Schur}, $\mathcal{M}$ is spectrally
  equivalent to the Schur complement $S = -\bar{B} \bar{A}^{-1}
  \bar{B}^T + \bar{C}$ of $\bar{\mathbb{A}}$. Furthermore, the Schur
  complement is invertible. To see this, note that $S = -\bar{B}
  \bar{A}^{-1} \bar{B}^T + \bar{C} = -B A^{-1} B^T$, where $A$ and $B$
  are the matrices given in \cref{eq:blocks}. The operator $A$ is
  symmetric positive definite, and $B$ is full rank by
  \cref{thm:stab_bh}, hence $S$ is invertible. Since $\bar{A}$ is
  symmetric positive definite by \cref{lem:specequiv_bara}, the result
  then follows by direct application
  of~\citep[Theorem~5.2]{Pestana:2015}. \qed
\end{proof}

The preconditioner in \cref{thm:global_sc_precon} is constructed based
on the fact that $\mathcal{M}$ is spectrally equivalent to the Schur
complement of the statically condensed linear system, as described by
\cref{cor:spec_equiv_sc_Schur}. In fact, $\mathcal{M}$ in
\cref{eq:precon-block-matrix-sc} can be replaced by any spectrally
equivalent operator $\mathcal{C}$, since $\mathcal{C}$ would then also
be spectrally equivalent to the Schur complement of the statically
condensed linear system. A particularly interesting choice for
$\mathcal{C}$ is discussed in the following.

\begin{theorem}[An optimal preconditioner for the statically condensed
  discrete Stokes problem based on element matrices]
  \label{thm:global_sc_precon_alt}
  Let $\bar{A}$, $\bar{B}$ and $\bar{C}$ be the matrices given in
  \cref{eq:system_S_comp} and let $\mathcal{C}$ be the negative
  block-diagonal of~$\bar{C}$:
  \begin{equation}
    \label{eq:def_of_C}
    \mathcal{C} =
    \begin{bmatrix}
      B_{pu}A_{uu}^{-1}B_{pu}^T & 0 \\
      0 & B_{\bar{p}u}A_{uu}^{-1}B_{\bar{p}u}^T
    \end{bmatrix}.
  \end{equation}
  Let $\bar{R}$ be an operator that is spectrally equivalent
  to~$\bar{A}$. Consider the preconditioned system
  \begin{equation}
    \label{eq:precon-block-matrix-sc-alt}
    \bar{\mathbb{P}}_{\mathcal{C}}^{-1}\bar{\mathbb{A}}\bar{\mathbb{U}} = \bar{\mathbb{P}}_{\mathcal{C}}^{-1}\bar{\mathbb{F}}
    \qquad \leftrightarrow \qquad
    \begin{bmatrix}
      \bar{R} & 0 \\ 0 & \mathcal{C}
    \end{bmatrix}^{-1}
    \begin{bmatrix}
      \bar{A} & \bar{B}^T \\
      \bar{B} & \bar{C}
    \end{bmatrix}
    \begin{bmatrix}
      \bar{u} \\
      P
    \end{bmatrix}
    =
    \begin{bmatrix}
      \bar{R} & 0 \\ 0 & \mathcal{C}
    \end{bmatrix}^{-1}
    \begin{bmatrix}
      \bar{L} \\ \bar{G}
    \end{bmatrix}.
  \end{equation}
  There exist positive constants $C_1, C_2, C_3, C_4$, independent of
  $h$, such that the negative and positive eigenvalues of
  $\bar{\mathbb{P}}_{\mathcal{C}}^{-1} \bar{\mathbb{A}}$ satisfy
  $\lambda \in [-C_1, -C_2]$ and $\lambda \in [C_3, C_4]$, respectively.
\end{theorem}
\begin{proof}
  It suffices to prove that $\mathcal{C}$ and $\mathcal{M}$ are
  spectrally equivalent. By minor modification of the proof of
  \cref{lem:spec_equiv_Schur}, by using the spectral equivalence of
  $a_h^{uu}(v_h, v_h)$ with $\tnorm{v_h}^2_{DG}$ (see
  \cref{eq:norm_equiv_ah_alt}) and the inf-sup conditions (see
  \cref{eq:infsupb1_dg,eq:infsupb2_dg}), it can by shown that $B_{pu}
  A_{uu}^{-1} B_{pu}^T$ and $M$ are spectrally equivalent and that
  $B_{\bar{p}u} A_{uu}^{-1} B_{\bar{p}u}^T$ and $\bar{M}$ are
  spectrally equivalent. It follows that $\mathcal{C}$ and
  $\mathcal{M}$ are spectrally equivalent. \qed
\end{proof}

%------------------------------------------------------------------------------
\subsection{Characterization of $\bar{A}$}
\label{ss:reduced_problem}

\Cref{thm:global_sc_precon,thm:global_sc_precon_alt} define optimal
preconditioners for the statically condensed discrete Stokes problem
provided we have an operator $\bar{R}$ that is spectrally equivalent
to~$\bar{A}$. To help in finding a suitable $\bar{R}$, we first
consider the properties of~$\bar{A}$.  The operator $\bar{A}$ is
obtained from the discretization of $\bar{a}_h(\bar{u}_h, \bar{v}_h)$
in \cref{eq:a_in_bits}. By \cref{lem:specequiv_bara} we know that
$\bar{A}$ is spectrally equivalent to the norm $\tnorm{\cdot}_h^2$. As
discussed, in for example~\cite{Lee:2017}, $\tnorm{\cdot}_h$ is a
$H^1$-like norm and the near-null space of $\bar{A}$ is spanned by
constant functions. This is a condition to successfully apply
multigrid-type solvers to~$\bar{A}$.  This motivates the use of
multigrid for the operator $\bar{R}$ that appears in
\cref{thm:global_sc_precon,thm:global_sc_precon_alt}.

%------------------------------------------------------------------------------
\subsection{Block symmetric Gauss--Seidel preconditioners}
\label{s:block_sgs}

\Cref{thm:global_sc_precon,thm:global_sc_precon_alt} introduce two
block-diagonal preconditioners. In practice, we see that the rates of
convergence of preconditioned iterative methods using the block
Jacobi-type preconditioners is typically improved upon by adding
off-diagonal blocks to the preconditioner. We therefore also consider
block symmetric Gauss--Seidel type preconditioners.

Let $\bar{\mathbb{A}}$ be the system matrix defined in
\cref{eq:system_S_comp}, $\mathcal{P}_D$ the block-diagonal of
$\bar{\mathbb{A}}$ and $\mathcal{P}_L$ a strictly lower triangular
block matrix such that
\begin{equation}
  \bar{\mathbb{A}} = \mathcal{P}_L + \mathcal{P}_D + \mathcal{P}_L^T.
\end{equation}
Furthermore, let $\mathcal{P}_M = {\rm bdiag}(-A_{\bar{u}u} A_{uu}^{-1}
A_{\bar{u}u}^T + A_{\bar{u} \bar{u}}, -M, -\bar{M})$. The block
symmetric Gauss--Seidel type preconditioners we consider in
\cref{s:numerical_examples} are:
\begin{equation}
  \label{eq:bSGS_C_M}
  \bar{\mathbb{P}}_{\mathcal{C}}^{SGS} = (\mathcal{P}_L + \mathcal{P}_D)\mathcal{P}_D^{-1}(\mathcal{P}_L^T + \mathcal{P}_D),
  \qquad
  \bar{\mathbb{P}}_{\mathcal{M}}^{SGS} = (\mathcal{P}_L + \mathcal{P}_M)\mathcal{P}_M^{-1}(\mathcal{P}_L^T + \mathcal{P}_M).
\end{equation}
In the numerical examples, the inverse of the first block of
$\mathcal{P}_D$ and $\mathcal{P}_M$ will be replaced by~$\bar{R}^{-1}$.

%------------------------------------------------------------------------------
\section{Numerical example}
\label{s:numerical_examples}

We now verify numerically the performance of the preconditioners
introduced in \cref{thm:global_sc_precon,thm:global_sc_precon_alt},
and the symmetric block Gauss--Seidel preconditioner in
\cref{eq:bSGS_C_M}.  We use a preconditioned MINRES solver, with AMG
(four multigrid V-cycles) for the operator~$\bar{R}^{-1}$. The inverse
of the pressure mass-matrix $\mathcal{M}$, and the spectrally
equivalent operator $\mathcal{C}$, are also approximated by four AMG
V-cycles. In both cases one application, pre and post, of a
Gauss--Seidel smoother is used. The MINRES iterations are terminated
once the relative true residual reaches a tolerance of~$10^{-8}$. We
consider unstructured simplicial meshes and unstructured quadrilateral
and hexahedral meshes. For simplex cells, we use a quadratic
polynomial approximation for $u_h$, $\bar{u}_h$ and $\bar{p}_h$, and a
linear polynomial approximation for~$p_h$. For meshes with
quadrilateral cells, we use a bi-quadratic polynomial approximation
for $u_h$, quadratic approximation of $\bar{u}_h$ and $\bar{p}_h$, and
a bilinear polynomial approximation for $p_h$. For meshes with
hexahedral cells, we use a tri-quadratic polynomial approximation for
$u_h$, bi-quadratic approximation of $\bar{u}_h$ and $\bar{p}_h$, and
a tri-linear polynomial approximation for~$p_h$. The stabilization
parameter is taken as $\alpha = 24$ in 2D and $\alpha = 40$ in~3D. The
formulation has been implemented in MFEM~\cite{mfem-library} with
solver support from PETSc~\cite{petsc-user-ref,petsc-web-page}. We use
classical algebraic multigrid via the BoomerAMG
library~\cite{Henson:2002}.

We consider lid-driven cavity flow in a square, $\Omega = [-1, 1]^2$,
and a cube, $\Omega = [0, 1]^3$. Dirichlet boundary conditions are
imposed on $\partial \Omega$. In two dimensions, $u = (1 - x_1^4, 0)$
on the boundary $x_2 = 1$ and the zero velocity vector on remaining
boundaries. In three dimensions we impose $u = (1 - \tau_1^4, (1 -
\tau_2^4)/10, 0)$, with $\tau_i = 2 x_i - 1$, on the boundary $x_3 =
1$ and the zero velocity vector on remaining boundaries.

\Cref{tab:liddriveniterations_tri} presents the iteration counts for
MINRES to converge for different levels of refinement on simplicial
meshes, and \cref{tab:liddriveniterations_quad} presents the iteration
counts for the quadrilateral and hexahedral mesh cases.
\begin{table}
  \begin{center}
    \begin{tabular}{c|cc||cc}
      \multicolumn{5}{c}{Two dimensions} \\
      \multicolumn{5}{c}{} \\
      \hline
      DOFs & $\bar{\mathbb{P}}_{\mathcal{M}}$ & $\bar{\mathbb{P}}_{\mathcal{M}}^{SGS}$
      & $\bar{\mathbb{P}}_{\mathcal{C}}$ & $\bar{\mathbb{P}}_{\mathcal{C}}^{SGS}$ \\
      \hline
      12012   & 136 & 73 & 94 & 89 \\
      47256   & 131 & 72 & 96 & 98 \\
      187440  & 134 & 69 & 96 & 102 \\
      746592  & 128 & 63 & 97 & 96 \\
      \hline
      \multicolumn{5}{c}{} \\
      \multicolumn{5}{c}{Three dimensions} \\
      \multicolumn{5}{c}{} \\
      \hline
      DOFs & $\bar{\mathbb{P}}_{\mathcal{M}}$ & $\bar{\mathbb{P}}_{\mathcal{M}}^{SGS}$
      & $\bar{\mathbb{P}}_{\mathcal{C}}$ & $\bar{\mathbb{P}}_{\mathcal{C}}^{SGS}$ \\
      \hline
      30128   & 230 & 150 & 122 & 139 \\
      229504  & 259 & 159 & 145 & 151 \\
      1789952 & 258 & 138 & 166 & 153 \\
      \hline
    \end{tabular}
  \end{center}
  \caption{Iteration counts for preconditioned MINRES for the relative
    true residual to reach a tolerance of $10^{-8}$ for the lid-driven
    cavity problem in two and three dimensions using unstructured
    simplicial meshes.}
  \label{tab:liddriveniterations_tri}
\end{table}
\begin{table}
  \begin{center}
    \begin{tabular}{c|cc||cc}
      \multicolumn{5}{c}{Two dimensions} \\
      \multicolumn{5}{c}{} \\
      \hline
      DOFs & $\bar{\mathbb{P}}_{\mathcal{M}}$ & $\bar{\mathbb{P}}_{\mathcal{M}}^{SGS}$
      & $\bar{\mathbb{P}}_{\mathcal{C}}$ & $\bar{\mathbb{P}}_{\mathcal{C}}^{SGS}$ \\
      \hline
      10956   & 109 & 56 & 84 & 79 \\
      43032   & 104 & 52 & 80 & 73 \\
      170544  & 104 & 47 & 80 & 68 \\
      679008  &  98 & 42 & 81 & 75 \\
      \hline
      \multicolumn{5}{c}{} \\
      \multicolumn{5}{c}{Three dimensions} \\
      \multicolumn{5}{c}{} \\
      \hline
      DOFs & $\bar{\mathbb{P}}_{\mathcal{M}}$ & $\bar{\mathbb{P}}_{\mathcal{M}}^{SGS}$
      & $\bar{\mathbb{P}}_{\mathcal{C}}$ & $\bar{\mathbb{P}}_{\mathcal{C}}^{SGS}$ \\
      \hline
      9152    & 130 & 79 & 88 & 82 \\
      66304   & 123 & 70 & 87 & 93 \\
      502784  & 114 & 57 & 85 & 78 \\
      \hline
    \end{tabular}
  \end{center}
  \caption{Iteration counts for preconditioned MINRES for the relative
    true residual to reach a tolerance of $10^{-8}$ for the lid-driven
    cavity problem in two and three dimensions using unstructured
    quadrilateral and structured hexahedral meshes.}
  \label{tab:liddriveniterations_quad}
\end{table}
It is clear that the iteration count does not grow with problem size in
all cases. Note that the diagonal preconditioners based on
$\mathcal{C}$, i.e., using only the blocks available from the system
matrix $\bar{\mathbb{A}}$ in \cref{eq:system_S_comp}, outperform the
preconditioners based on $\mathcal{M}$ consisting of the element
pressure mass-matrix $M$ and the scaled facet pressure mass-matrix
$\bar{M}$. In the case of the symmetric block Gauss-Seidel
preconditioners, there is no gain in using $\mathcal{C}$
over~$\mathcal{M}$.

We have observed that switching from left-preconditioned MINRES to
right-preconditioned GMRES can improve the iteration count
substantially. For example, in the case of the
$\bar{\mathbb{P}}_{\mathcal{C}}^{SGS}$ preconditioner for a
three-dimensional simplicial grid with 1789952 DOFs, the iteration
count for right-preconditioned GMRES is only 37 (compared with 153
iterations for left-preconditioned MINRES).

%------------------------------------------------------------------------------
\section{Conclusions}
\label{sec:conclusions}

We have developed, analyzed and numerically tested two new
block-diagonal preconditioners for the statically condensed linear
system for a hybridized discontinuous Galerkin method for the Stokes
equations. In particular, we proved and showed numerically that the
preconditioners are optimal in that preconditioned systems can be
solved to a specified tolerance in an iteration count that is
independent of the problem size. This makes the preconditioner
suitable for very large systems, and especially for problems in which
pointwise satisfaction of the continuity equation is important since
the considered method has this valuable property. Discretizations of
the Stokes problem using a hybridized discontinuous Galerkin method
permit static condensation; cell degrees-of-freedom can be eliminated
locally, resulting in significantly reduced number of globally coupled
degrees-of-freedom.  This does however complicate the analysis, and
our analysis addresses the form and structure of a condensed problem.

%------------------------------------------------------------------------------
\appendix
%------------------------------------------------------------------------------
\section{Auxiliary results}
\label{s:aux_results}

We provide here some auxiliary results used in analyzing the
preconditioners.

Defining $a_h^{uu}(u_h, v_h) := a_h((u_h, 0), (v_h, 0))$, since
$\tnorm{v_h}_{DG} = \tnorm{(v_h, 0)}_v$, a consequence of
\cref{eq:norm_equiv_ah} is:
\begin{equation}
  \label{eq:norm_equiv_ah_alt}
  c_a^s\tnorm{v_h}_{DG}^2 \le a_h^{uu}(v_h, v_h) \le c_a^b\tnorm{v_h}_{DG}^2.
\end{equation}
Applying \citep[Proposition~10]{Hansbo:2002} to a single cell $K$, the
following inf-sup condition holds:
\begin{equation}
  \label{eq:infsupb1_dg_K}
  \beta_{DG}^K \norm{q_h}_{K} \le \sup_{v_h \in V_h(K)}
  \frac{(q_h, \nabla \cdot v_h)_{K}}{\tnorm{v_h}_{DG(K)}}
  \quad \forall q_{h} \in P_{k-1}(K),
\end{equation}
where $\beta_{DG}^K > 0$ is a constant independent of $h$, $V_h(K) :=
\sbr{P_k(K)}^d$ and $\tnorm{v_h}_{DG(K)}^2 := \norm{\nabla v_h}^2_{K} +
\alpha h_K^{-1} \norm{v_h}^2_{\partial K}$. It follows that
\begin{equation}
  \label{eq:infsupb1_dg}
  \beta_{DG} \norm{q_h}_{\Omega} \le \sup_{v_h \in V_h}
  \frac{(q_h, \nabla\cdot v_h)_{\mathcal{T}}}{\tnorm{v_h}_{DG}},
\end{equation}
where $\beta_{DG} := \min_{K \in \mathcal{T}} \beta_{DG}^K$. Since
$\tnorm{v_h}_{DG} = \tnorm{(v_h, 0)}_v$, it is easy to see from
\cref{eq:infsuponK} that
\begin{equation}
  \label{eq:infsupb2_dg}
  \bar{\beta}_{DG} \norm{\bar{q}_h}_{p} \le \sup_{v_h \in V_h}
  \frac{ \langle v_h \cdot n, \bar{q}_h \rangle_{\partial\mathcal{T}}}{\tnorm{ v_h }_{DG}},
\end{equation}
where $\bar{\beta}_{DG} > 0$ is a constant independent of~$h$.

%------------------------------------------------------------------------------
\bibliographystyle{spbasic}
\bibliography{references}
%------------------------------------------------------------------------------
\end{document}